\documentclass[a4paper,oneside,10pt]{article}%
\usepackage{amsmath}
\usepackage{amsfonts}
\usepackage{amssymb}
\usepackage{graphicx}
\usepackage{color}
\usepackage[square,numbers,sort&compress]{natbib}%
\providecommand{\U}[1]{\protect\rule{.1in}{.1in}}

\newtheorem{theorem}{Theorem}[section]

\newtheorem{corollary}[theorem]{Corollary}

\newtheorem{definition}[theorem]{Definition}
\newtheorem{assumption}[theorem]{Assumption}
\newtheorem{example}[theorem]{Example}

\newtheorem{lemma}[theorem]{Lemma}

\newtheorem{proposition}[theorem]{Proposition}
\newtheorem{remark}[theorem]{Remark}

\newenvironment{proof}[1][Proof]{\noindent\textbf{#1.} }{\ \rule{0.5em}{0.5em}}
\numberwithin{equation}{section}

\begin{document}

\title{The Neyman-Pearson lemma for convex expectations }
\author{Chuanfeng Sun\thanks{School of Mathematical Sciences, University of Jinan,
Jinan, Shandong 250022, P.R. China. e-mail: sms\_suncf@ujn.edu.cn. This
research is partially supported by the National Natural Science Foundation of
China (No. 11701214), the Natural Science Foundation of Shandong Province (No.
ZR2017BA032).} \quad Shaolin Ji\thanks{Zhongtai Institute of Finance, Shandong
University, Jinan, Shandong 250100, PR China. e-mail: jsl@sdu.edu.cn. This
research is supported by National Natural Science Foundation of China (No.
11571203), the Programme of Introducing Talents of Discipline to Universities
of China (No. B12023).}}
\date{}
\maketitle

\textbf{Abstract}. We study the Neyman-Pearson theory for convex expectations
(convex risk measures) on $L^{\infty}(\mu)$. Without assuming that the level
sets of penalty functions are weakly compact, a new approach different from
the convex duality method is proposed to find a representative pair $(Q^{\ast
},P^{\ast})$ such that the optimal tests are just the classical Neyman-Pearson
tests between the representative probabilities $Q^{\ast}$ and $P^{\ast}$. The
key observation is that the feasible test set is compact in the weak$^{\ast}$
topology by a generalized result of Banach-Alaoglu theorem. Then the minimax
theorem can be applied and the representative probability $Q^{\ast}$ is found
first. Secondly, under the probability $Q^{\ast}$, we find the representative
probability measure $P^{\ast}$ by solving a dual problem. Finally, we apply
our results to a shortfall risk minimizing problem in an incomplete financial market.



\section{Introduction}

The classical Neyman-Pearson lemma gives the most powerful test for
discriminating between two probability measures and has important applications
in various fields (see \cite{fe}, \cite{le-ro}).

It is well known that many phenomena need to be explored by nonlinear
probabilities or expectations. In 1954, Choquet \cite{r3} extended the
probability measure to the capacity and gave a nonlinear integral named after
him. The coherent risk measure was proposed by Artzner et al. \cite{r5} and
the $g$-expectation was initiated by Peng \cite{r4} in 1999. F\"{o}llmer and
Schied \cite{r6} generalized the coherent risk measure to the convex risk
measure in 2002. Divergence risk measures were considered by Ben-Tal and
Teboulle \cite{BT} under the name of optimized certainty equivalents.

Along with the development of the above concepts, several nonlinear versions
of Neyman-Pearson lemma have also been established. In 1973, Huber and
Strassen \cite{{r7}} studied the Neyman-Pearson lemma for capacities.
Cvitani\'{c} and Karatzas \cite{r1} extended the classical Neyman-Pearson
theory for testing composite hypotheses versus composite alternatives in 2001.
Later Schied \cite{r2} gave a Neyman-Pearson lemma for law-invariant coherent
risk measures and robust utility functionals. Ji and Zhou \cite{r8} studied
hypothesis tests for $g$-probabilities in 2010. Rudloff and Karatzas \cite{r9}
studied composite hypotheses by using convex duality in 2010. Apart from their
own theoretical value, the nonlinear versions of Neyman-Pearson lemma have
been found to have many applications especially in finance. For instance,
F\"{o}llmer and Leukert \cite{fo-le-1999} and \cite{fo-le-2000} studied the
quantile hedging and efficient hedging which minimizes the shortfall risks in
an incomplete financial market. Rudloff \cite{r10} found a self-financing
strategy that minimize the convex risk of the shortfall using convex duality method.

In fact, the composite hypotheses testing problem in \cite{r1} can also be
seen as discriminating between two sublinear expectations. A natural
generalization is how to discriminating between two convex expectations. In
this paper, we mainly investigate the Neyman-Pearson lemma for convex
expectations on $L^{\infty}(\mu)$. In our context, the definition of convex
expectation is essentially equivalent to that of convex risk measure (see
Definition \ref{DEF-convex expectation}). For two given convex expectations
$\rho_{1}$, $\rho_{2}$ on $L^{\infty}(\mu)$ and a significance level $\alpha$,
we want to find an optimal test $X^{\ast}$ which minimizes the expectation of
Type II error with respect to $\rho_{2}$, among all tests that keep the
expectation of Type I error with respect to $\rho_{1}$ below the given
acceptable significance level $\alpha\in(0,1)$. In other words, we study the
following problem:%
\begin{equation}
\text{minimize}\,\rho_{2}(1-X),\label{prelimilary}%
\end{equation}
over the set $\mathcal{X}_{\alpha}=\{X\in L^{\infty}(\mu):0\leq X\leq1,$
$\rho_{1}(X)\leq\alpha\}$. In order to study the Neyman-Pearson-type
optimization problems more conveniently, in this paper, we actually embed
problem (\ref{prelimilary}) into a broader problem: for two given random
variables $K_{1}$ and $K_{2}$ belonging to $L^{\infty}(\mu)$ such that $0\leq
K_{1}<K_{2}$,
\begin{equation}
\text{minimize}\quad\rho_{2}(K_{2}-X),\label{Priliminary-1}%
\end{equation}
over the set $\mathcal{X}_{\alpha}=\{X:K_{1}\leq X\leq K_{2},\rho_{1}%
(X)\leq\alpha,X\in L^{\infty}(\mu)\}$.

The main purpose of studying Neyman-Pearson lemma is to find the form of the
optimal test. An interesting question is whether there exists a representative
pair of probabilities $(Q^{\ast},P^{\ast})$ such that the optimal test for
problem (\ref{prelimilary}) is just the optimal test between the simple
hypotheses $Q^{\ast}$ and $P^{\ast}$. In most literatures, the convex duality
method is employed to study the nonlinear Neyman-Pearson lemma
(Neyman-Pearson-type optimization problems) and the corresponding pair of
simple hypotheses is found. For example, without assuming the set of densities
which generate the sublinear expectation is weakly compact, Cvitani\'{c} and
Karatzas \cite{r1} studied the Neyman-Pearson lemma for sublinear
expectations. To minimize the shortfall risk in an incomplete market,
F\"{o}llmer, Leukert \cite{fo-le-2000} and Rudloff \cite{r10} chose a specific
convex risk measure and the convex risk measure on $L^{1}(\mu)$ respectively.
They solved the corresponding Neyman-Pearson-type optimization problems in
which the sets of densities that generate the convex risk measures are weakly compact.

To solve problem (\ref{Priliminary-1}), we can not apply the convex duality
method as in \cite{r1}. The reason is that this method needs to determine the
representative pair $(Q^{\ast},P^{\ast})$ at the same time and the additional
penalty function terms in the representation of convex expectations make this
approach impossible. So in this paper, We propose first finding the
probability $Q^{\ast}$ and then looking for the probability $P^{\ast}$ under
the fixed probability $Q^{\ast}$. The main difficulty in finding $Q^{\ast}$ is
that we only assume that the level sets of penalty functions are closed under
the $\mu$-a.e. convergence which is similar to the assumption in \cite{r1}.
Under this assumption, the set of densities which generate a convex
expectation on $L^{\infty}(\mu)$ is not weakly compact in general and the
minimax theorem seems inapplicable. The key to solving this difficulty is that
we find the feasible set $\mathcal{X}_{\alpha}$ is compact in the weak$^{\ast
}$ topology $\sigma(L^{\infty},L^{1})$ by a generalized result of
Banach-Alaoglu Theorem. Based on this observation, the minimax theorem can be
applied and the representative probability $Q^{\ast}$ for $\rho_{2}$ is found.
Under the fixed probability $Q^{\ast}$, finding the probability $P^{\ast}$
directly by the convex duality method is technically complicated. By solving
its dual problem, we also find the representative probability measure
$P^{\ast}$ for $\rho_{1}$. Thus, the optimal tests for convex expectations on
$L^{\infty}(\mu)$ are just the classical Neyman-Pearson tests between a fixed
representative pair $(Q^{\ast},P^{\ast})$.

It is obvious that a convex expectation on $L^{1}(\mu)$ is also a convex
expectation on $L^{\infty}(\mu)$. So the Neyman-Pearson lemma for convex
expectations on $L^{1}(\mu)$ is a natural inference of the Neyman-Pearson
lemma for convex expectations on $L^{\infty}(\mu)$ (see Corollary
\ref{NP lemma-L1}).

Finally, we apply our results to a shortfall risk minimizing problem in an
incomplete financial market. The shortfall risk is measured by the convex
expectation of the shortfall. For a partially hedged contingent claim $H\in
L^{\infty}(\mu)$, we consider the convex expectation on $L^{\infty}(\mu)$ and
solve this minimizing problem by Theorem \ref{main-result} and the classical
Neyman-Pearson lemma. For $H\in L^{1}(\mu)$, we need to study the following
problem: for $K_{1}$ and $K_{2}$ belonging to $L^{1}(\mu)$ such that $0\leq
K_{1}<K_{2}$,
\begin{equation}
\text{minimize}\quad\rho_{2}(K_{2}-X),\label{Priliminary-2}%
\end{equation}
over the set $\mathcal{X}_{\alpha}=\{X:K_{1}\leq X\leq K_{2},\rho_{1}%
(X)\leq\alpha,X\in L^{1}(\mu)\}$, where $\rho_{1}$ and $\rho_{2}$ are two
given convex expectations on $L^{1}(\mu)$. We find that similar ideas for
solving problem (\ref{Priliminary-1}) can be used to solve problem
(\ref{Priliminary-2}). Since the set of densities which generate a convex
expectation on $L^{1}(\mu)$ is weakly compact which greatly simplifies the
proof, we only put this result in the appendix and give a brief proof.

This paper is organized as follows: In Section 2, we give some preliminaries
and formulate the simple hypothesis testing problem for convex expectations on
$L^{\infty}(\mu)$. The existence of the optimal tests is derived in section 3.
In section 4, we obtain the form of the optimal tests. An application is given
to illustrate our main results in section 5. Finally, in the appendix we show
that if convex expectations are continuous from above, then Assumption
\ref{assumption} holds naturally and solve problem (\ref{Priliminary-2}) for
convex expectations on $L^{1}(\mu)$.

\section{Preliminaries and Problem Formulation}

Let $(\Omega,\mathcal{F},\mu)$ be a probability space and $\mathcal{M}$ be the
set of probability measures on $(\Omega,\mathcal{F})$ that are absolutely
continuous with respect to $\mu$. $P$ and $Q$ are probability measures and
their Radon-Nikodym derivatives $\frac{dP}{d\mu}$ and $\frac{dQ}{d\mu}$ are
denoted as $G_{P}$ and $H_{Q}$ respectively.

\begin{definition}
\label{DEF-convex expectation}A mapping $\rho$: $L^{\infty}(\mu)\rightarrow
\mathbb{R}$ is called a convex expectation on $L^{\infty}(\mu)$\ if for any
$X,Y\in L^{\infty}(\mu)$, we have

(i) Monotonicity: If $X\geq Y$, then $\rho(X)\geq\rho(Y)$;

(ii) Invariance: If $c$ is a constant, then $\rho(X+c)=\rho(X)+c$;

(iii) Convexity: If $\lambda\in[0, 1]$, then $\rho\big(\lambda X+(1-\lambda
)Y\big)\leq\lambda\rho(X)+(1-\lambda)\rho(Y)$.
\end{definition}

In the above definition, If $L^{\infty}(\mu)$ is replaced by $L^{1}(\mu)$,
then we can define the convex expectation on $L^{1}(\mu)$ similarly. Obviously
a convex expectation on $L^{1}(\mu)$ is also a convex expectation on
$L^{\infty}(\mu)$. Unless specifically stated, a convex expectation refers to
a convex expectation on $L^{\infty}(\mu)$ in this paper.

Note that if we define $\rho^{\prime}(X)=\rho(-X)$, then $\rho^{\prime}$ is a
convex risk measure.

\begin{definition}
\label{continuous from below} We call a convex expectation $\rho$ is
continuous from below iff for any sequence $\{X_{n}\}_{n\geq1}\subset
L^{\infty}(\mu)$ increases to some $X\in L^{\infty}(\mu)$, then $\rho
(X_{n})\rightarrow\rho(X)$.
\end{definition}

The following theorem comes from Theorem 6 and Proposition 7 in \cite{r6}.

\begin{theorem}
\label{representation} If a convex expectation $\rho$ is continuous from
below, then

i) For any $X\in L^{\infty}(\mu)$,
\begin{equation}
\rho(X)=\sup\limits_{P\in\mathcal{M}}\big(E_{P}[X]-\rho^{\ast}(P)\big),
\end{equation}
where $\rho^{\ast}$ is the penalty function of $\rho$ and $\rho^{\ast}%
(P)=\sup\limits_{X\in L^{\infty}(\mu)}\big(E_{P}[X]-\rho(X)\big)$.

ii) For any bounded sequence $\{X_{n}\}_{n\geq1}\subset L^{\infty}(\mu)$, if
$X_{n}$ converges to some $X\in L^{\infty}(\mu)$ in probability, then
$\rho(X)\leq\liminf\limits_{n\rightarrow\infty}\rho(X_{n})$.
\end{theorem}

\subsection{Composite hypotheses and alternatives}

Given two convex expectations $\rho_{1}$ and $\rho_{2}$, by (i) of Theorem
\ref{representation},
\[
\rho_{1}(X)=\sup\limits_{P\in\mathcal{M}}(E_{P}[X]-\rho_{1}^{\ast}%
(P))\quad\text{and}\quad\rho_{2}(X)=\sup\limits_{Q\in\mathcal{M}}%
(E_{Q}[X]-\rho_{2}^{\ast}(Q)),
\]
where $\rho_{i}^{\ast}$ is the penalty function of $\rho_{i}$ for $i=1,2$.

If we denote
\[
\mathcal{P}=\{P:P\in\mathcal{M},\rho_{1}^{\ast}(P)<\infty\}\quad
\text{and}\quad\mathcal{Q}=\{Q:Q\in\mathcal{M},\rho_{2}^{\ast}(Q)<\infty\},
\]
then $\mathcal{P}$ and $\mathcal{Q}$ are nonempty convex sets and
\[
\rho_{1}(X)=\sup\limits_{P\in\mathcal{P}}(E_{P}[X]-\rho_{1}^{\ast}%
(P))\quad\text{and}\quad\rho_{2}(X)=\sup\limits_{Q\in\mathcal{Q}}%
(E_{Q}[X]-\rho_{2}^{\ast}(Q)).
\]

Suppose that $\mathcal{P\cap Q=\varnothing}$. Now we want to discriminate
$\mathcal{P}$ (composite hypotheses) against $\mathcal{Q}$ (composite
alternatives) for a significance level $\alpha$. 
Then we shall look for a randomized test $X^{\ast}$ which minimizes the
maximum error (Type II)
\begin{equation}
\sup\limits_{Q\in\mathcal{Q}}(E_{Q}[1-X]-\rho_{2}^{\ast}(Q)),
\label{composite problem}%
\end{equation}
over all randomized tests $X$ such that%
\[
\sup\limits_{P\in\mathcal{P}}(E_{P}[X]-\rho_{1}^{\ast}(P))\leq\alpha.
\]
\

It is worth pointing out that problem (\ref{composite problem}) is a natural
extension of the problem (2.6) in \cite{r1}. For a possible probability
measure $P$ (resp. $Q$) and a candidate randomized test $X$, only the
expectation $E_{P}[X]$ (resp. $E_{Q}[X]$) was taken into account in \cite{r1}.
Different from this in \cite{r1}, we generalize to the case that considering
$E_{P}[X]-\rho_{1}^{\ast}(P)$ (resp. $E_{Q}[X]-\rho_{2}^{\ast}(Q)$) where
$\rho_{1}^{\ast}(P)$ (resp. $\rho_{2}^{\ast}(Q)$) may be understood as a
``weight" for a probability measure $P$ (resp. $Q$) over $\mathcal{P}$ (resp.
$\mathcal{Q}$).

Note that $\sup\limits_{P\in\mathcal{P}}E_{P}[\cdot]$ and $\sup\limits_{Q\in
\mathcal{Q}}E_{Q}[\cdot]$ define two sublinear expectations; $\sup
\limits_{P\in\mathcal{P}}(E_{P}[\cdot]-\rho_{1}^{\ast}(P))$ and $\sup
\limits_{Q\in\mathcal{Q}}(E_{Q}[\cdot]-\rho_{2}^{\ast}(Q))$ define two convex
expectations. Then, from another point of view, the problem (2.6) in \cite{r1}
(resp. our problem (\ref{composite problem})) can be understood as
discriminating between two sublinear expectations (resp. convex expectations).
In other words, our problem (\ref{composite problem}) can be rewritten as
problem (\ref{prelimilary}):
\[
\text{minimize}\quad\rho_{2}(1-X)
\]
over the set $\mathcal{X}_{\alpha}=\{X:\Omega\rightarrow\lbrack0,1],\;\rho
_{1}(X)\leq\alpha\}.$

\subsection{A general problem}

In order to investigate Neyman-Pearson lemma and Neyman-Pearson-type
optimization problems together, we study the following more general problem.

Given two convex expectations $\rho_{1}$ and $\rho_{2}$, for a significance
level $\alpha$ and two random variables $K_{1}$ and $K_{2}$ belonging to
$L^{\infty}(\mu)$ such that $0\leq K_{1}<K_{2}$, we want to
\begin{equation}
\text{minimize}\quad\rho_{2}(K_{2}-X),\label{initial-problem}%
\end{equation}
over the set
\[
\mathcal{X}_{\alpha}=\{X:K_{1}\leq X\leq K_{2},\rho_{1}(X)\leq\alpha,X\in
L^{\infty}(\mu)\}.
\]
Without loss of generality, we assume $\rho_{1}(K_{1})\leq\alpha\leq\rho
_{1}(K_{2})$. Note that if $K_{1}=0$ and $K_{2}=1$, then the above problem
becomes problem (\ref{prelimilary}).

For simplicity, we still call $X\in\mathcal{X}_{\alpha}$ a test for our
general problem (\ref{initial-problem}).

\begin{definition}
We call $X^{\ast}$ the optimal test of (\ref{initial-problem}) if $X^{\ast}%
\in\mathcal{X}_{\alpha}$ and
\begin{equation}
\rho_{2}(K_{2}-X^{\ast})=\inf\limits_{X\in\mathcal{X}_{\alpha}}\rho_{2}%
(K_{2}-X).
\end{equation}

\end{definition}

Under some mild assumptions on $\mathcal{P}$ and $\mathcal{Q}$, we shall prove
that an optimal test exists and has a similar form of the optimal tests for
the classical Neyman-Pearson theory.

\section{The existence of the optimal test}

Set $\beta=\inf\limits_{X\in\mathcal{X}_{\alpha}}\rho_{2}(K_{2}-X)$. The
following result shows that the optimal test exists. \

\begin{theorem}
\label{existence} If $\rho_{1}$ and $\rho_{2}$ are convex expectations
continuous from below, then the optimal test of (\ref{initial-problem}) exists.
\end{theorem}

\begin{proof}
Take a sequence $\{X_{n}\}_{n\geq1}\subset\mathcal{X}_{\alpha}$ such that
\[
\rho_{2}(K_{2}-X_{n})<\beta+\frac{1}{2^{n}}.
\]
By the Koml\'{o}s theorem, there exist a subsequence $\{X_{n_{i}}\}_{i\geq1}$
of $\{X_{n}\}_{n\geq1}$ and a random variable $X^{\ast}$ such that \
\begin{equation}
\lim_{k\rightarrow\infty}\frac{1}{k}\sum_{i=1}^{k}X_{n_{i}}=X^{\ast},\quad
\mu-a.e..
\end{equation}
Since for any $n$, $K_{1}\leq X_{n}\leq K_{2}$, we have $K_{1}\leq X^{\ast
}\leq K_{2}$, $\mu$-a.e.. By (ii) of Theorem \ref{representation},
\[
\rho_{1}(X^{\ast})\leq\liminf_{k\rightarrow\infty}\rho_{1}(\frac{1}{k}%
\sum_{i=1}^{k}X_{n_{i}})\leq\liminf_{k\rightarrow\infty}\frac{1}{k}\sum
_{i=1}^{k}\rho_{1}(X_{n_{i}})\leq\alpha,
\]
which leads to $X^{\ast}\in\mathcal{X}_{\alpha}$. On the other hand,
\[
\rho_{2}(K_{2}-X^{\ast})\leq\liminf_{k\rightarrow\infty}\frac{1}{k}\sum
_{i=1}^{k}\rho_{2}(K_{2}-X_{n_{i}})\leq\beta+\lim_{k\rightarrow\infty}\frac
{1}{k}=\beta.
\]
Thus,
\[
\rho_{2}(K_{2}-X^{\ast})=\inf\limits_{X\in\mathcal{X}_{\alpha}}\rho_{2}%
(K_{2}-X).
\]
This completes the proof.
\end{proof}

\section{The form of the optimal test}

Note that
\[
\inf\limits_{X\in\mathcal{X}_{\alpha}}\rho_{2}(K_{2}-X)=\inf\limits_{X\in
\mathcal{X}_{\alpha}}\sup\limits_{Q\in\mathcal{Q}}\big(E_{Q}[K_{2}-X]-\rho
_{2}^{\ast}(Q)\big).
\]
Then $X^{\ast}$ is the optimal test of (\ref{initial-problem}) if and only if
it is the optimal test of the problem:
\begin{equation}
\text{minimize}\quad\sup\limits_{Q\in\mathcal{Q}}\big(E_{Q}[K_{2}-X]-\rho
_{2}^{\ast}(Q)\big), \label{extended-problem}%
\end{equation}
over $\mathcal{X}_{\alpha}$.

Now we focus on solving problem (\ref{extended-problem}). Denote the level
sets of penalty functions $\rho_{1}^{\ast}$ and $\rho_{2}^{\ast}$ as
\[
\mathcal{G}_{c}=\{G_{P}:P\in\mathcal{P},\text{\ }\rho_{1}^{\ast}(P)\leq
c\}\quad\text{and}\quad\mathcal{H}_{c}=\{H_{Q}:Q\in\mathcal{Q},\text{\ }%
\rho_{2}^{\ast}(Q)\leq c\},
\]
where $c$ is a constant. Since $\rho_{1}^{\ast}$ and $\rho_{2}^{\ast}$ are
convex functions on $\mathcal{M}$, then both $\mathcal{G}_{c}$ and
$\mathcal{H}_{c}$ are convex sets.

Since $K_{1}$ and $K_{2}$ belong to $L^{\infty}(\mu)$, we denote the least
upper bound of them by $M$.

\begin{assumption}
\label{assumption} There exist two constants $u>\max\{0, M-\rho_{1}(0)+1\}$
and $v>\max\{0, M-\rho_{2}(0)+1\}$ such that $\mathcal{G}_{u}$ and
$\mathcal{H}_{v}$ are both closed under the $\mu$-a.e. convergence.
\end{assumption}

Since the penalty function of the sublinear expectation takes only the values
$0$ and $+\infty$, for sublinear case, Assumption \ref{assumption} is equal to
require $\{G_{P}:P\in P\}$ and $\{H_{Q}:Q\in Q\}$ are both closed under the
$\mu$-a.e. convergence, which is similar as the assumption given by
Cvitani\'{c} and Karatzas in \cite{r1}. In Appendix, we show that if $\rho
_{1}$ and $\rho_{2}$ are continuous from above, then Assumption
\ref{assumption} holds naturally.

\subsection{The existence of a representative probability $Q^{\ast}$}

In this subsection, we want to find a representative probability $Q^{\ast}\in\mathcal{Q}$
such that
\[
\inf\limits_{X\in\mathcal{X}_{\alpha}}\sup\limits_{Q\in\mathcal{Q}}%
\big(E_{Q}[K_{2}-X]-\rho_{2}^{\ast}(Q)\big)=\inf\limits_{X\in\mathcal{X}%
_{\alpha}}E_{Q^{\ast}}[K_{2}-X]-\rho_{2}^{\ast}(Q^{\ast})\text{.}%
\]
If such a $Q^{\ast}$ exists, then for any optimal test $X^{\ast}$ of
(\ref{initial-problem}), we have
\[
\sup\limits_{Q\in\mathcal{Q}}\big(E_{Q}[K_{2}-X^{\ast}]-\rho_{2}^{\ast
}(Q)\big)=\inf\limits_{X\in\mathcal{X}_{\alpha}}E_{Q^{\ast}}[K_{2}-X]-\rho
_{2}^{\ast}(Q^{\ast}),
\]
which leads to $E_{Q^{\ast}}[K_{2}-X^{\ast}]=\inf\limits_{X\in\mathcal{X}%
_{\alpha}}E_{Q^{\ast}}[K_{2}-X]$.

\begin{theorem}
\label{minimax-result} If $\rho_{1}$ and $\rho_{2}$ are convex expectations
continuous from below and Assumption \ref{assumption} holds, then there exists
$Q^{\ast}\in\mathcal{Q}$ such that for any optimal test $X^{\ast}$ of
(\ref{initial-problem}), we have
\begin{equation}
E_{Q^{\ast}}[K_{2}-X^{\ast}]=\inf\limits_{X\in\mathcal{X}_{\alpha}}E_{Q^{\ast
}}[K_{2}-X].
\end{equation}

\end{theorem}

Before proving Theorem \ref{minimax-result}, we first give some lemmas.

\begin{lemma}
\label{Fatou-property-2} For any sequence $\{Q_{n}\}_{n\geq1}\subset
\mathcal{M}$, if $H_{Q_{n}}$ converges to some $H_{Q_{0}}$ under $L^{1}(\mu)$
norm, then
\begin{equation}
\inf\limits_{X\in\mathcal{X}_{\alpha}}E_{Q_{0}}[K_{2}-X]\geq\limsup
_{n\rightarrow\infty}\inf\limits_{X\in\mathcal{X}_{\alpha}}E_{Q_{n}}[K_{2}-X].
\end{equation}

\end{lemma}

\begin{proof}
For any $X\in\mathcal{X}_{\alpha}$, we have
\[
E_{Q_{0}}[K_{2}-X]=\lim\limits_{n\to\infty}E_{Q_{n}}[K_{2}-X]\geq
\limsup\limits_{n\to\infty}\inf\limits_{X\in\mathcal{X}_{\alpha}}E_{Q_{n}%
}[K_{2}-X].
\]
Then
\[
\inf\limits_{X\in\mathcal{X}_{\alpha}}E_{Q_{0}}[K_{2}-X]\geq\limsup
\limits_{n\to\infty}\inf\limits_{X\in\mathcal{X}_{\alpha}}E_{Q_{n}}[K_{2}-X].
\]
This completes the proof.
\end{proof}

\begin{lemma}
\label{weak*-compact} If $\rho_{1}$ is a convex expectation continuous from
below, then $\mathcal{X}_{\alpha}$ is compact in the weak$^{*}$ topology
$\sigma(L^{\infty}(\mu), L^{1}(\mu))$.
\end{lemma}

\begin{proof}
Define $\phi(Y)=\sup\limits_{X\in\mathcal{X}_{\alpha}}E_{\mu}[X\cdot Y]$,
where $Y\in L^{1}(\mu)$. Then $\phi$ is a sublinear function on $L^{1}(\mu)$
and dominated by $M||\cdot||_{L^{1}(\mu)}$. Set
\begin{equation}
\hat{\mathcal{X}}_{\alpha}=\{X\in L^{\infty}(\mu):E_{\mu}[X\cdot Y]\leq
\phi(Y)\ \text{for any}\ Y\in L^{1}(\mu)\}.
\end{equation}
By a generalized result of Banach-Alaoglu theorem (Theorem 4.2 of chapter I in
\cite{r11}), $\hat{\mathcal{X}}_{\alpha}$ is compact in the weak$^{\ast}$
topology $\sigma(L^{\infty}(\mu),L^{1}(\mu))$. Then we only need to show
\[
\mathcal{X}_{\alpha}=\hat{\mathcal{X}}_{\alpha}.
\]
Since $\mathcal{X}_{\alpha}\subset\hat{\mathcal{X}}_{\alpha}$ obviously, in
the next, we will show $\hat{\mathcal{X}}_{\alpha}\subset\mathcal{X}_{\alpha}$.

Firstly, for any $\hat{X}\in\hat{\mathcal{X}}_{\alpha}$, we show that
$K_{1}\leq\hat{X}\leq K_{2}$, $\mu$-a.e.. If there exists $\hat{X}\in
\hat{\mathcal{X}}_{\alpha}$ such that $\mu(\{\omega:\hat{X}(\omega
)<K_{1}\})\not =0$, then there will exist a constant $\epsilon>0$ such that
$\mu(A)\not =0$, where $A=\{\omega:\hat{X}(\omega)\leq K_{1}-\epsilon\}$. For
any $X\in\mathcal{X}_{\alpha}$, since $\hat{X}\leq K_{1}-\epsilon$ on $A$, we
have $\hat{X}\leq X-\epsilon$ on $A$. Let $h_{A}=-\dfrac{I_{A}}{\mu(A)}$.
Then
\[
E_{\mu}[\hat{X}h_{A}]=-\dfrac{1}{\mu(A)}E_{\mu}[\hat{X}I_{A}]\geq-\dfrac
{1}{\mu(A)}E_{\mu}[(X-\epsilon)I_{A}]=E_{\mu}[Xh_{A}]+\epsilon.
\]
Due to $X$ can be taken in $\mathcal{X}_{\alpha}$ arbitrarily, we have
\[
E_{\mu}[\hat{X}h_{A}]\geq\sup\limits_{X\in\mathcal{X}_{\alpha}}E_{\mu}%
[Xh_{A}]+\epsilon>\sup\limits_{X\in\mathcal{X}_{\alpha}}E_{\mu}[Xh_{A}%
]=\phi(h_{A}).
\]
Since $h_{A}\in L^{1}(\mu)$, it contradicts with $\hat{X}\in\hat{\mathcal{X}%
}_{\alpha}$. Thus, $\hat{X}\geq K_{1}$, $\mu$-a.e.. Similarly, we can prove
$\hat{X}\leq K_{2}$, $\mu$-a.e..

Next, we show for any $\hat{X}\in\hat{\mathcal{X}}_{\alpha}$, $\rho_{1}%
(\hat{X})\leq\alpha$. Since $\hat{X}\in\hat{\mathcal{X}}_{\alpha}$, for any
$P\in\mathcal{P}$,
\[
E_{P}[\hat{X}]=E_{\mu}[\hat{X}G_{P}]\leq\sup\limits_{X\in\mathcal{X}_{\alpha}%
}E_{\mu}[XG_{P}]=\sup\limits_{X\in\mathcal{X}_{\alpha}}E_{P}[X].
\]
Then
\[%
\begin{array}
[c]{r@{}l}%
\rho_{1}(\hat{X})= & \sup\limits_{P\in\mathcal{P}}\big(E_{P}[\hat{X}]-\rho
_{1}^{\ast}(P)\big)\\
\leq & \sup\limits_{P\in\mathcal{P}}\sup\limits_{X\in\mathcal{X}_{\alpha
^{\ast}}}\big(E_{P}[X]-\rho_{1}^{\ast}(P)\big)\\
= & \sup\limits_{X\in\mathcal{X}_{\alpha}}\sup\limits_{P\in\mathcal{P}%
}\big(E_{P}[X]-\rho_{1}^{\ast}(P)\big)\\
= & \sup\limits_{X\in\mathcal{X}_{\alpha}}\rho_{1}(X)\leq\alpha.
\end{array}
\]
Thus, $\hat{X}\in\mathcal{X}_{\alpha}$.
\end{proof}

\begin{remark}
If $\rho_{1}$ degenerates to be a sublinear expectation, the above result can
also be found in \cite{r10}.
\end{remark}

\begin{lemma}
\label{aux-minimax} If $\rho_{1}$ and $\rho_{2}$ are convex expectations
continuous from below, then
\begin{equation}
\inf\limits_{X\in\mathcal{X}_{\alpha}}\sup\limits_{Q\in\mathcal{Q}}%
\big(E_{Q}[K_{2}-X]-\rho_{2}^{\ast}(Q)\big)=\sup\limits_{Q\in\mathcal{Q}}%
\inf\limits_{X\in\mathcal{X}_{\alpha}}\big(E_{Q}[K_{2}-X]-\rho_{2}^{\ast
}(Q)\big). \label{3.3}%
\end{equation}

\end{lemma}

\begin{proof}
Since for each $X\in\mathcal{X}_{\alpha}$, $E_{Q}[K_{2}-X]-\rho_{2}^{\ast}(Q)$
is a concave function on $\mathcal{Q}$ and for each $Q\in\mathcal{Q}$,
$E_{Q}[K_{2}-X]-\rho_{2}^{\ast}(Q)$ is a linear continuous function on
$L^{\infty}(\mu)$, with $\mathcal{X}_{\alpha}$ is compact in the weak$^{\ast}$
topology $\sigma(L^{\infty}(\mu),L^{1}(\mu))$, then by the minimax theorem
(Refer to Theorem 3.2 of chapter I in \cite{r11}), the equation (\ref{3.3}) holds.
\end{proof}

The following lemma shows that $\rho^{\ast}$ is lower semi-continuous.

\begin{lemma}
\label{Fatou-property-1} If $\rho$ is a convex expectation continuous from
below, for any sequence $\{Q_{n}\}_{n\geq1}\subset\mathcal{M}$ and $Q_{0}%
\in\mathcal{M}$ such that $H_{Q_{n}}$ converges to $H_{Q_{0}}$, $\mu$-a.e.,
then
\[
\rho^{\ast}(Q_{0})\leq\liminf\limits_{n\rightarrow\infty}\rho^{\ast}(Q_{n}).
\]

\end{lemma}

\begin{proof}
Set
\[
L_{+}^{\infty}(\mu)=\{X\in L^{\infty}(\mu):X\geq0\}.
\]
Then $\rho^{\ast}$ can be redefined as
\[
\rho^{\ast}(Q)=\sup\limits_{X\in L_{+}^{\infty}(\mu)}\big(E_{Q}[X]-\rho
(X)\big),
\]
since $E_{Q}[X]-\rho(X)=E_{Q}[X+m]-\rho(X+m)$ for any $Q\in\mathcal{M}$, $X\in
L^{\infty}(\mu)$ and positive real number $m$.

Take $J_{k}=\inf\limits_{n\geq k}H_{Q_{n}}$. Then $\{J_{k}\}_{k\geq1}$ is an
increasing sequence and $H_{Q_{0}}=\sup\limits_{k\geq1}J_{k}$. We have
\[%
\begin{array}
[c]{r@{}l}%
\rho^{\ast}(Q_{0})= & \sup\limits_{X\in L_{+}^{\infty}(\mu)}\big(E_{\mu
}[X(\sup\limits_{k\geq1}J_{k})]-\rho(X)\big)\\
= & \sup\limits_{k\geq1}\sup\limits_{X\in L_{+}^{\infty}(\mu)}\big(E_{\mu
}[XJ_{k}]-\rho(X)\big)\\
= & \sup\limits_{k\geq1}\sup\limits_{X\in L_{+}^{\infty}(\mu)}\big(E_{\mu
}[\inf\limits_{n\geq k}(XH_{Q_{n}})]-\rho(X)\big)\\
\leq & \sup\limits_{k\geq1}\sup\limits_{X\in L_{+}^{\infty}(\mu)}%
\inf\limits_{n\geq k}\big(E_{Q_{n}}[X]-\rho(X)\big)\\
\leq & \sup\limits_{k\geq1}\inf\limits_{n\geq k}\sup\limits_{X\in
L_{+}^{\infty}(\mu)}\big(E_{Q_{n}}[X]-\rho(X)\big)\\
= & \liminf\limits_{n\rightarrow\infty}\rho^{\ast}(Q_{n}).
\end{array}
\]
This completes the proof.
\end{proof}

\begin{lemma}
\label{auxiliary-result} If $\rho_{1}$ and $\rho_{2}$ are convex expectations
continuous from below and Assumption \ref{assumption} holds, then there exists
$Q^{\ast}\in\mathcal{Q}$ such that
\begin{equation}
\inf\limits_{X\in\mathcal{X}_{\alpha}}E_{Q^{\ast}}[K_{2}-X]-\rho_{2}^{\ast
}(Q^{\ast})=\sup\limits_{Q\in\mathcal{Q}}\inf\limits_{X\in\mathcal{X}_{\alpha
}}\big(E_{Q}[K_{2}-X]-\rho_{2}^{\ast}(Q)\big). \label{equation-3.5}%
\end{equation}

\end{lemma}

\begin{proof}
Take a positive constant $0<\epsilon<1$ and a sequence $\{Q_{n}\}_{n\geq
1}\subset\mathcal{Q}$ such that
\[
\inf\limits_{X\in\mathcal{X}_{\alpha}}E_{Q_{n}}[K_{2}-X]-\rho_{2}^{\ast}%
(Q_{n})\geq\gamma-\frac{\epsilon}{2^{n}},
\]
where $\gamma=\sup\limits_{Q\in\mathcal{Q}}\inf\limits_{X\in\mathcal{X}%
_{\alpha}}\big(E_{Q}[K_{2}-X]-\rho_{2}^{\ast}(Q)\big)$. By Lemma
\ref{aux-minimax},
\[
\gamma=\inf\limits_{X\in\mathcal{X}_{\alpha}}\sup\limits_{Q\in\mathcal{Q}%
}\big(E_{Q}[K_{2}-X]-\rho_{2}^{\ast}(Q)\big)=\inf\limits_{X\in\mathcal{X}%
_{\alpha}}\rho_{2}(K_{2}-X).
\]
Since
\[
\rho_{2}(0)\leq\inf\limits_{X\in\mathcal{X}_{\alpha}}\rho_{2}(K_{2}-X),
\]
then $\rho_{2}(0)\leq\gamma$. For any $n$,
\[
M-\rho_{2}^{\ast}(Q_{n})\geq\inf\limits_{X\in\mathcal{X}_{\alpha}}E_{Q_{n}%
}[K_{2}-X]-\rho_{2}^{\ast}(Q_{n})\geq\gamma-\frac{\epsilon}{2^{n}}\geq
\gamma-\epsilon,
\]
which leads to
\[
\rho_{2}^{\ast}(Q_{n})\leq M-\gamma+\epsilon\leq M-\rho_{2}(0)+1.
\]
For $v$ defined in Assumption \ref{assumption}, we have $\rho_{2}^{\ast}%
(Q_{n})\leq v$ which implies $\{H_{Q_{n}}\}_{n\geq1}\subset\mathcal{H}_{v}$.

By the Koml\'{o}s Theorem, there exist a subsequence $\{Q_{n_{i}}\}_{i\geq1}$
of $\{Q_{n}\}_{n\geq1}$ and a random variable $H^{\ast}\in L^{1}(\mu)$ such
that
\[
\lim_{k\rightarrow\infty}\frac{1}{k}\sum_{i=1}^{k}H_{Q_{n_{i}}}=H^{\ast}%
,\quad\mu-a.e..
\]
Since $\mathcal{H}_{v}$ is a convex set and closed under the $\mu$-a.e.
convergence, then $H^{\ast}\in\mathcal{H}_{v}$. Denote $Q^{\ast}$ as the
corresponding probability measure of $H^{\ast}$. Since
\[
\lim_{k\rightarrow\infty}\frac{1}{k}\sum_{i=1}^{k}H_{Q_{n_{i}}}=H^{\ast}%
,\quad\mu-a.e.
\]
and
\[
1=E_{\mu}[H^{\ast}]=\lim_{k\rightarrow\infty}E_{\mu}[\frac{1}{k}\sum_{i=1}%
^{k}H_{Q_{n_{i}}}],
\]
we have $\{\frac{1}{k}\sum_{i=1}^{k}H_{Q_{n_{i}}}\}_{k\geq1}$ converges to
$H^{\ast}$ under $L^{1}(\mu)$ norm. By Lemma \ref{Fatou-property-2} and Lemma
\ref{Fatou-property-1},
\[%
\begin{array}
[c]{r@{}l}
& \inf\limits_{X\in\mathcal{X}_{\alpha}}E_{Q^{\ast}}[K_{2}-X]-\rho_{2}^{\ast
}(Q^{\ast})\\
\geq & \limsup\limits_{k\rightarrow\infty}\inf\limits_{X\in\mathcal{X}%
_{\alpha}}E_{\mu}[(K_{2}-X)(\dfrac{1}{k}\sum\limits_{i=1}^{k}H_{Q_{n_{i}}%
})]-\liminf\limits_{k\rightarrow\infty}\rho_{2}^{\ast}(\dfrac{1}{k}%
\sum\limits_{i=1}^{k}Q_{n_{i}})\\
\geq & \limsup\limits_{k\rightarrow\infty}\inf\limits_{X\in\mathcal{X}%
_{\alpha}}\dfrac{1}{k}\sum\limits_{i=1}^{k}\big(E_{Q_{n_{i}}}[(K_{2}%
-X)]-\rho_{2}^{\ast}(Q_{n_{i}})\big)\\
\geq & \limsup\limits_{k\rightarrow\infty}\dfrac{1}{k}\sum\limits_{i=1}%
^{k}\inf\limits_{X\in\mathcal{X}_{\alpha}}\big(E_{Q_{n_{i}}}[(K_{2}%
-X)]-\rho_{2}^{\ast}(Q_{n_{i}})\big)\\
\geq & \lim\limits_{k\rightarrow\infty}(\gamma-\dfrac{\epsilon}{k})=\gamma.
\end{array}
\]
Since $Q^{\ast}\in\mathcal{Q}$, we have
\[
\inf\limits_{X\in\mathcal{X}_{\alpha}}E_{Q^{\ast}}[K_{2}-X]-\rho_{2}^{\ast
}(Q^{\ast})=\sup\limits_{Q\in\mathcal{Q}}\inf\limits_{X\in\mathcal{X}_{\alpha
}}\big(E_{Q}[K_{2}-X]-\rho_{2}^{\ast}(Q)\big).
\]
This completes the proof.
\end{proof}

Summarizing all the lemmas above, we obtain the following proof of Theorem
\ref{minimax-result}:

\begin{proof}
By Lemma \ref{auxiliary-result}, there exists $Q^{\ast}\in\mathcal{Q}$ such
that
\[
\inf\limits_{X\in\mathcal{X}_{\alpha}}E_{Q^{\ast}}[K_{2}-X]-\rho_{2}^{\ast
}(Q^{\ast})=\sup\limits_{Q\in\mathcal{Q}}\inf\limits_{X\in\mathcal{X}_{\alpha
}}\big(E_{Q}[K_{2}-X]-\rho_{2}^{\ast}(Q)\big).
\]
If $X^{\ast}$ is the optimal test of (\ref{initial-problem}), then
\[
\sup\limits_{Q\in\mathcal{Q}}\big(E_{Q}[K_{2}-X^{\ast}]-\rho_{2}^{\ast
}(Q)\big)=\inf\limits_{X\in\mathcal{X}_{\alpha}}\sup\limits_{Q\in\mathcal{Q}%
}\big(E_{Q}[K_{2}-X]-\rho_{2}^{\ast}(Q)\big).
\]
By Lemma \ref{aux-minimax},
\[
\inf\limits_{X\in\mathcal{X}_{\alpha}}\sup\limits_{Q\in\mathcal{Q}}%
\big(E_{Q}[K_{2}-X]-\rho_{2}^{\ast}(Q)\big)=\sup\limits_{Q\in\mathcal{Q}}%
\inf\limits_{X\in\mathcal{X}_{\alpha}}\big(E_{Q}[K_{2}-X]-\rho_{2}^{\ast
}(Q)\big).
\]
Thus,
\[
\inf\limits_{X\in\mathcal{X}_{\alpha}}E_{Q^{\ast}}[K_{2}-X]-\rho_{2}^{\ast
}(Q^{\ast})=\sup\limits_{Q\in\mathcal{Q}}\big(E_{Q}[K_{2}-X^{\ast}]-\rho
_{2}^{\ast}(Q)\big).
\]
Since
\[
\inf\limits_{X\in\mathcal{X}_{\alpha}}E_{Q^{\ast}}[K_{2}-X]-\rho_{2}^{\ast
}(Q^{\ast})\leq E_{Q^{\ast}}[K_{2}-X^{\ast}]-\rho_{2}^{\ast}(Q^{\ast})\leq
\sup\limits_{Q\in\mathcal{Q}}\big(E_{Q}[K_{2}-X^{\ast}]-\rho_{2}^{\ast
}(Q)\big),
\]
then
\[
E_{Q^{\ast}}[K_{2}-X^{\ast}]-\rho_{2}^{\ast}(Q^{\ast})=\inf\limits_{X\in
\mathcal{X}_{\alpha}}E_{Q^{\ast}}[K_{2}-X]-\rho_{2}^{\ast}(Q^{\ast}),
\]
i.e.,
\[
E_{Q^{\ast}}[K_{2}-X^{\ast}]=\inf\limits_{X\in\mathcal{X}_{\alpha}}E_{Q^{\ast
}}[K_{2}-X].
\]
This completes the proof.
\end{proof}

\begin{example}
\label{example-helpless} Consider the probability space $(\Omega
,\mathcal{F},\mu)$, where $\Omega=\{0,1\}$, $\mathcal{F}=\{\emptyset
,\{0\},\{1\},\Omega\}$. Set
\[
\mu(\omega)=\Bigg\{%
\begin{array}
[c]{l@{}c}%
\frac{1}{2}, & \quad\text{if }\omega=0,\\
\frac{1}{2}, & \quad\text{if }\omega=1,
\end{array}
\text{\ \ and\ \ }Q_{0}(\omega)=\Bigg\{%
\begin{array}
[c]{l@{}c}%
\frac{3}{4}, & \quad\text{if }\omega=0,\\
\frac{1}{4}, & \quad\text{if }\omega=1
\end{array}
.
\]
Let $K_{1}=0$, $K_{2}=1$, $\alpha=\frac{1}{2}$, $\rho_{1}(X)=E_{\mu}[X]$ and
$\rho_{2}(X)=\ln E_{Q_{0}}[e^{X}]$. We solve problem (\ref{initial-problem}).
Let $Q=qI_{\{0\}}+(1-q)I_{\{1\}}$, where $0\leq q\leq1$. Then
\[
\rho_{2}^{\ast}(Q)=E_{Q_{0}}[\frac{dQ}{dQ_{0}}\ln\frac{dQ}{dQ_{0}}]=q\ln
q+(1-q)\ln(1-q)-q\ln3+2\ln2.
\]
Let $X=x_{0}I_{\{0\}}+x_{1}I_{\{1\}}$, where $0\leq x_{0},$ $x_{1}\leq1$. If
$X\in\mathcal{X}_{\alpha}$, then $\frac{1}{2}x_{0}+\frac{1}{2}x_{1}\leq
\frac{1}{2}$, i.e., $x_{0}\leq1-x_{1}$. When $q=\frac{3}{e+3}$, $\sup\limits_{Q\in\mathcal{Q}}\inf
\limits_{X\in\mathcal{X}_{\alpha}}E_{Q}[1-X]-\rho_{2}^{\ast}(Q)$ attains its
maximum on $\mathcal{Q}$. Thus,
\[
Q^{\ast}=\frac{3}{e+3}I_{\{0\}}+\frac{e}{e+3}I_{\{1\}}\text{ and }X^{\ast
}=I_{\{0\}}.
\]

\end{example}

\subsection{The existence of a representative probability $P^{\ast}$}

In the rest of this paper, $Q^{\ast}$ is always\ the probability measure found
in Theorem \ref{minimax-result}. Define
\[
\gamma_{\alpha}=\inf\limits_{X\in\mathcal{X}_{\alpha}}E_{Q^{\ast}}[K_{2}-X].
\]
If $\gamma_{\alpha}=0$, then it is trivial and the optimal test $X^{\ast
}=K_{2}$, $Q^{\ast}$-a.e.. In the following, we only consider the case
$\gamma_{\alpha}>0$.

\begin{lemma}
\label{dual-problem} If $\gamma_{\alpha}>0$, $\rho_{1}$ and $\rho_{2}$ are
convex expectations continuous from below and Assumption \ref{assumption}
holds, then for any optimal test $X^{\ast}$ of (\ref{initial-problem}), we
have $X^{\ast}\in\mathcal{X}^{\gamma_{\alpha}}$ and
\begin{equation}
\rho_{1}(X^{\ast})=\inf\limits_{X\in\mathcal{X}^{\gamma_{\alpha}}}\rho
_{1}(X)=\alpha, \label{2nd-dual-problem}%
\end{equation}
where $\mathcal{X}^{\gamma_{\alpha}}=\{X:E_{Q^{\ast}}[K_{2}-X]\leq
\gamma_{\alpha},K_{1}\leq X\leq K_{2},X\in L^{\infty}(\mu)\}$.
\end{lemma}

\begin{proof}
$X^{\ast}\in\mathcal{X}^{\gamma_{\alpha}}$ comes from Theorem
\ref{minimax-result}. For any $X\in\mathcal{X}_{\alpha}$, if $\rho
_{1}(X)<\alpha$, we claim $E_{Q^{\ast}}[K_{2}-X]>\gamma_{\alpha}$. If not,
then there will exist a test $X^{\prime}\in\mathcal{X}_{\alpha}$ such that
$\rho_{1}(X^{\prime})<\alpha$ and
\[
E_{Q^{\ast}}[K_{2}-X^{\prime}]=\gamma_{\alpha}.
\]
Set
\[
\rho_{1}(X^{\prime})=\alpha^{\prime}<\alpha
\]
and
\[
X^{\prime\prime}=(X^{\prime}+\alpha-\alpha^{\prime})\wedge K_{2}.
\]
By the definition of convex expectation,
\[
\rho_{1}(X^{\prime\prime})\leq\rho_{1}(X^{\prime}+\alpha-\alpha^{\prime}%
)=\rho_{1}(X^{\prime})+\alpha-\alpha^{\prime}=\alpha,
\]
which implies that $X^{\prime\prime}\in\mathcal{X}_{\alpha}$. As
$X^{\prime\prime}\in\mathcal{X}_{\alpha}$ and $X^{\prime\prime}\geq X^{\prime
}$ we have $E_{Q^{\ast}}[K_{2}-X^{\prime\prime}]=E_{Q^{\ast}}[K_{2}-X^{\prime
}]$, i.e., $E_{Q^{\ast}}[X^{\prime\prime}]=E_{Q^{\ast}}[X^{\prime}]$. Set
$A=\{X^{\prime}\not =K_{2}\}$. Since
\[
X^{\prime\prime}-X^{\prime}\geq0\quad\text{and}\quad E_{Q^{\ast}}%
[X^{\prime\prime}-X^{\prime}]=0,
\]
we have $X^{\prime\prime}=X^{\prime}$, $Q^{\ast}$-a.e., which implies that
$Q^{\ast}(A)=0$ and $X^{\prime}=K_{2}$, $Q^{\ast}$-a.e.. Then $\gamma_{\alpha
}=0$ which contradicts with $\gamma_{\alpha}>0$.

Thus, for any $X\in\mathcal{X}^{\gamma_{\alpha}}$, we have $\rho_{1}%
(X)\geq\alpha$. With $\rho_{1}(X^{\ast})=\alpha$, the result holds.
\end{proof}

\begin{theorem}
\label{Second-minimax-result} Suppose that $\gamma_{\alpha}>0$, $\rho_{1}$ and
$\rho_{2}$ are convex expectations continuous from below and Assumption
\ref{assumption} holds. Then there exists $P^{\ast}\in\mathcal{P}$ such that
for any optimal test $X^{\ast}$ of (\ref{initial-problem}),
\[
E_{P^{\ast}}[X^{\ast}]=\inf\limits_{X\in\mathcal{X}^{\gamma_{\alpha}}%
}E_{P^{\ast}}[X].
\]

\end{theorem}

\begin{proof}
Set $Y=K_{2}-X$, $Y^{\ast}=K_{2}-X^{\ast}$ and
\[
\mathcal{Y}_{\gamma_{\alpha}}=\{Y:E_{Q^{\ast}}[Y]\leq\gamma_{\alpha},0\leq
Y\leq K_{2}-K_{1},Y\in L^{\infty}(\mu)\}.
\]
By Lemma \ref{dual-problem},
\[
\rho_{1}(K_{2}-Y^{\ast})=\inf\limits_{Y\in\mathcal{Y}_{\gamma_{\alpha}}}%
\rho_{1}(K_{2}-Y),
\]
i.e.,
\begin{equation}
\sup\limits_{P\in\mathcal{P}}\big(E_{P}[K_{2}-Y^{\ast}]-\rho_{1}^{\ast
}(P)\big)=\inf\limits_{Y\in\mathcal{Y}_{\gamma_{\alpha}}}\sup\limits_{P\in
\mathcal{P}}\big(E_{P}[K_{2}-Y]-\rho_{1}^{\ast}(P)\big). \label{3.4}%
\end{equation}
Applying similar analysis as in Lemma \ref{weak*-compact}, we obtain that
$\mathcal{Y}_{\gamma_{\alpha}}$ is compact in the topology $\sigma(L^{\infty
}(\mu),L^{1}(\mu))$. By the minimax theorem,
\begin{equation}
\inf\limits_{Y\in\mathcal{Y}_{\gamma_{\alpha}}}\sup\limits_{P\in\mathcal{P}%
}\big(E_{P}[K_{2}-Y]-\rho_{1}^{\ast}(P)\big)=\sup\limits_{P\in\mathcal{P}}%
\inf\limits_{Y\in\mathcal{Y}_{\gamma_{\alpha}}}\big(E_{P}[K_{2}-Y]-\rho
_{1}^{\ast}(P)\big). \label{3.5}%
\end{equation}

Now we prove that there exists a probability measure $P^{\ast}\in\mathcal{P}$
such that
\begin{equation}
\inf\limits_{Y\in\mathcal{Y}_{\gamma_{\alpha}}}\big(E_{P^{\ast}}[K_{2}%
-Y]-\rho_{1}^{\ast}(P^{\ast})\big)=\sup\limits_{P\in\mathcal{P}}%
\inf\limits_{Y\in\mathcal{Y}_{\gamma_{\alpha}}}\big(E_{P}[K_{2}-Y]-\rho
_{1}^{\ast}(P)\big). \label{3.6}%
\end{equation}
If we replace $X$ by $Y$, $\mathcal{X}_{\alpha}$ by $\mathcal{Y}%
_{\gamma_{\alpha}}$, $P$ by $Q$ and $\rho_{1}^{\ast}$ by $\rho_{2}^{\ast}$ in
(\ref{equation-3.5}), then (\ref{equation-3.5}) becomes (\ref{3.6}). Using the
same proof method as in Lemma \ref{auxiliary-result}, we deduce that
\ \ (\ref{3.6}) holds.

By (\ref{3.4}), (\ref{3.5}) and (\ref{3.6}),
\[
\inf\limits_{Y\in\mathcal{Y}_{\gamma_{\alpha}}}\big(E_{P^{\ast}}[K_{2}%
-Y]-\rho_{1}^{\ast}(P^{\ast})\big)=\sup\limits_{P\in\mathcal{P}}%
\big(E_{P}[K_{2}-Y^{\ast}]-\rho_{1}^{\ast}(P)\big).
\]
Since
\[%
\begin{array}
[c]{r@{}l}%
\inf\limits_{Y\in\mathcal{Y}_{\gamma_{\alpha}}}E_{P^{\ast}}[K_{2}-Y]-\rho
_{1}^{\ast}(P^{\ast})\leq & E_{P^{\ast}}[K_{2}-Y^{\ast}]-\rho_{1}^{\ast
}(P^{\ast})\\
\leq & \sup\limits_{P\in\mathcal{P}}\big(E_{P}[K_{2}-Y^{\ast}]-\rho_{1}^{\ast
}(P)\big),
\end{array}
\]
we have
\[
E_{P^{\ast}}[K_{2}-Y^{\ast}]-\rho_{1}^{\ast}(P^{\ast})=\inf\limits_{Y\in
\mathcal{Y}_{\gamma_{\alpha}}}E_{P^{\ast}}[K_{2}-Y]-\rho_{1}^{\ast}(P^{\ast
}).
\]
Thus,
\[
E_{P^{\ast}}[K_{2}-Y^{\ast}]=\inf\limits_{Y\in\mathcal{Y}_{\gamma_{\alpha}}%
}E_{P^{\ast}}[K_{2}-Y],
\]
i.e.,
\[
E_{P^{\ast}}[X^{\ast}]=\inf\limits_{X\in\mathcal{X}^{\gamma_{\alpha}}%
}E_{P^{\ast}}[X].
\]
This completes the proof.
\end{proof}

\begin{example}
\label{example-discrete} Consider the probability space $(\Omega
,\mathcal{F},\mu)$, where $\Omega$, $\mathcal{F}$ and $\mu$ are defined as the
same as in Example \ref{example-helpless}. Set $K_{1}=0$, $K_{2}=1$,
$\alpha=\ln(e+3)-2\ln2$, $\rho_{1}(X)=\ln E_{P_{0}}[e^{X}]$ and $\rho
_{2}(X)=E_{\mu}[X]$, where
\[
P_{0}(\omega)=\Bigg\{%
\begin{array}
[c]{l@{}c}%
\frac{1}{4}, & \quad\text{if }\omega=0,\\
\frac{3}{4}, & \quad\text{if }\omega=1.
\end{array}
\]
We solve problem (\ref{initial-problem}). It is easy to check that
\[
\inf\limits_{X\in\mathcal{X}_{\alpha}}E_{\mu}(1-X)=\frac{1}{2},
\]
i.e., $\gamma_{\alpha}=\frac{1}{2}$. By Lemma \ref{dual-problem}, to solve
problem (\ref{initial-problem}) is equivalent to solve the following problem:
\begin{equation}
\text{minimize}\quad\rho_{1}(X), \label{problem-initial-extend}%
\end{equation}
over the set $\mathcal{X}^{\gamma_{\alpha}}=\{X:E_{\mu}[X]\geq\frac{1}%
{2},0\leq X\leq1\}$. Let $X=x_{0}I_{\{0\}}+x_{1}I_{\{1\}}$, where $0\leq
x_{0},x_{1}\leq1$. If $X\in\mathcal{X}^{\gamma_{\alpha}}$, then $x_{0}%
\geq1-x_{1}$. Let $P=pI_{\{0\}}+(1-p)I_{\{1\}}$, where $0\leq p\leq1$. Then
\[
\rho_{1}^{\ast}(P)=E_{P_{0}}[\frac{dP}{dP_{0}}\ln\frac{dP}{dP_{0}}]=2\ln2+p\ln
p+(1-p)\ln(1-p)-(1-p)\ln3.
\]
When $p=\frac{e}{e+3}$, $\sup\limits_{P\in\mathcal{P}}\inf\limits_{X\in\mathcal{X}^{\gamma_{\alpha}}}%
E_{P}[X]-\rho_{1}^{\ast}(P)$ attains its maximum on $\mathcal{P}$. Thus,
\[
P^{\ast}=\frac{e}{e+3}I_{\{0\}}+\frac{3}{e+3}I_{\{1\}}\text{ \ and \ }X^{\ast
}=I_{\{0\}}.
\]

\end{example}

\subsection{Main result\label{main results}}

\begin{theorem}
\label{main-result} If $\rho_{1}$ and $\rho_{2}$ are convex expectations
continuous from below and Assumption \ref{assumption} holds, then there exist
$P^{\ast}\in\mathcal{P}$ and $Q^{\ast}\in\mathcal{Q}$ such that for any
optimal test $X^{\ast}$ of (\ref{initial-problem}), it can be expressed as
\begin{equation}
X^{\ast}=K_{2}I_{\{H_{Q^{\ast}}>zG_{P^{\ast}}\}}+BI_{\{H_{Q^{\ast}%
}=zG_{P^{\ast}}\}}+K_{1}I_{\{H_{Q^{\ast}}<zG_{P^{\ast}}\}},\quad\mu-a.e.,
\label{form}%
\end{equation}
where $z\in\lbrack0,+\infty)\cup\{+\infty\}$ and $B$ is a random variable
taking values in the random interval $[K_{1},K_{2}]$.
\end{theorem}

\begin{proof}
We divide our proof into two cases:

i) The case $\gamma_{\alpha}>0$. By Theorem \ref{Second-minimax-result},
$X^{\ast}$ is the optimal test of the following problem:
\[
\text{minimize}\quad E_{P^{\ast}}[X],
\]
over the set $\mathcal{X}^{\gamma_{\alpha}}=\{X:E_{Q^{\ast}}[K_{2}%
-X]\leq\gamma_{\alpha},K_{1}\leq X\leq K_{2},X\in L^{\infty}(\mu)\}$. Set
\[
Z^{\ast}=\frac{K_{2}-X^{\ast}}{K_{2}-K_{1}},\, Z=\frac{K_{2}-X}{K_{2}-K_{1}%
},\,\gamma_{\alpha}^{\prime}=\frac{\gamma_{\alpha}}{E_{Q^{\ast}}[K_{2}-K_{1}%
]},\,\frac{d\hat{P}}{dP^{\ast}}=\frac{K_{2}-K_{1}}{E_{P^{\ast}}[K_{2}-K_{1}%
]}\,\text{and}\,\frac{d\hat{Q}}{dQ^{\ast}}=\frac{K_{2}-K_{1}}{E_{Q^{\ast}%
}[K_{2}-K_{1}]}.
\]
Then $Z^{\ast}$ is the optimal test of the problem:
\begin{equation}
\text{maximize}\quad E_{\hat{P}}[Z], \label{linear-form}%
\end{equation}
over the set $\mathcal{Z}_{\gamma_{\alpha}^{\prime}}=\{Z:E_{\hat{Q}}%
[Z]\leq\gamma_{\alpha}^{\prime},0\leq Z\leq1,Z\in L^{\infty}(\mu)\}$.

By the classical Neyman-Pearson lemma (see \cite{r1} or Theorem A.30 in
\cite{r12}), any optimal test $Z^{\ast}$ of (\ref{linear-form}) has the form
\begin{equation}
Z^{\ast}=I_{\{z^{\prime}H_{\hat{Q}}<G_{\hat{P}}\}}+B^{\prime}\cdot
I_{\{z^{\prime}H_{\hat{Q}}=G_{\hat{P}}\}},\quad\mu-a.e.
\end{equation}
for some constant $z^{\prime}\geq0$ and random variable $B^{\prime}$ taking
values in the interval $[0,1]$. Since
\[
\frac{d\hat{P}}{dP^{\ast}}=\frac{K_{2}-K_{1}}{E_{P^{\ast}}[K_{2}-K_{1}]}%
\quad\text{and}\quad\frac{d\hat{Q}}{dQ^{\ast}}=\frac{K_{2}-K_{1}}{E_{Q^{\ast}%
}[K_{2}-K_{1}]},
\]
if we take (with conventions $+\infty=\frac{1}{0}$ and $0=\frac{0}{0}$)
\[
B=K_{2}-(K_{2}-K_{1})B^{\prime}\quad\text{and}\quad z=\frac{E_{Q^{\ast}}%
[K_{2}-K_{1}]}{z^{\prime}E_{P^{\ast}}[K_{2}-K_{1}]},
\]
then $z\in(0, +\infty)\cup\{+\infty\}$ and
\begin{equation}
\label{representation-form}X^{\ast}=K_{2}I_{\{H_{Q^{\ast}}>zG_{P^{\ast}}%
\}}+BI_{\{H_{Q^{\ast}}=zG_{P^{\ast}}\}}+K_{1}I_{\{H_{Q^{\ast}}<zG_{P^{\ast}%
}\}},\quad\mu-a.e..
\end{equation}

ii) The case $\gamma_{\alpha}=0$. For this case, $X^{\ast}=K_{2}$, $Q^{\ast}%
$-a.e.. This is a special case of (\ref{representation-form}) when $z$ equals
$0$.
\end{proof}

In the next, we consider the case that $\rho_{1}$ and $\rho_{2}$ are two
convex expectations defined on $L^{1}(\mu)$. It is obvious that $\rho_{1}$ and
$\rho_{2}$ are also two convex expectations on $L^{\infty}(\mu)$. Thus, for
our problem (\ref{initial-problem}) we have the following corollary by Theorem
\ref{main-result}:

\begin{corollary}
\label{NP lemma-L1}If $\rho_{1}$ and $\rho_{2}$ are two convex expectations
defined on $L^{1}(\mu)$ space, then the optimal test of (\ref{initial-problem}%
) exists and has the form as in Theorem \ref{main-result}.
\end{corollary}

\begin{example}
\label{main-example} Except $\rho_{2}(X)=\ln E_{Q_{0}}[e^{X}]$ where $Q_{0}$
is defined as in Example \ref{example-helpless}, all the notations in this
example are defined as the same as in Example \ref{example-discrete}. We solve
problem (\ref{initial-problem}).

Denote $\mathcal{Z}=\{X:0\leq X\leq1,E_{\mu}[X]\leq\frac{1}{2}\}$. By Example
\ref{example-discrete}, we have $\sup\limits_{X\in\mathcal{X}_{\alpha}}E_{\mu
}[X]=\frac{1}{2}$. Then $\mathcal{X}_{\alpha}\subset\mathcal{Z}$ and
\begin{equation}
\inf\limits_{X\in\mathcal{Z}}\rho_{2}(1-X)\leq\inf\limits_{X\in\mathcal{X}%
_{\alpha}}\rho_{2}(1-X). \label{au-example}%
\end{equation}
Take $\hat{X}=I_{\{0\}}$. By Example \ref{example-helpless},
\[
\rho_{2}(1-\hat{X})=\inf\limits_{X\in\mathcal{Z}}\rho_{2}(1-X).
\]
Since $\hat{X}\in\mathcal{X}_{\alpha}$, with (\ref{au-example}), we have
\[
\rho_{2}(1-\hat{X})=\inf\limits_{X\in\mathcal{X}_{\alpha}}\rho_{2}(1-X),
\]
which implies $I_{\{0\}}$ is the optimal test. Furthermore, if we take
$Q^{\ast}=\frac{3}{e+3}I_{\{0\}}+\frac{e}{e+3}I_{\{1\}}$ and $P^{\ast}%
=\frac{e}{e+3}I_{\{0\}}+\frac{3}{e+3}I_{\{1\}}$ as in Examples
\ref{example-helpless} and \ref{example-discrete}, then
\[
I_{\{0\}}=I_{\{\frac{3}{e}H_{Q^{\ast}}>G_{P^{\ast}}\}}.
\]

\end{example}

\section{Application}

In a financial market, if an investor does not have enough initial wealth,
then he may fail to (super-) hedge an contingent claim and will face some
shortfall risk. In this case, we need a criterion expressing the investor's
attitude towards the shortfall risk (see \cite{fo-le-1999,fo-le-2000,
fo-2002-b, r6}). F\"{o}llmer and Leukert \cite{fo-le-2000} use the expectation
of the shortfall weighted by the loss function as a shortfall risk measure. In
this section, we use a general measure, the convex risk measure, to evaluate
the shortfall and consequently minimize such a shortfall risk.

In more details, we adopt the same financial market model as in
\cite{fo-le-2000}. The discounted price process of the underlying asset is
described as a semimartingale $S=(S_{t})_{t\in\lbrack0,T]}$ on a complete
probability space $(\Omega,\mathcal{F},\mu)$. The information structure is
given by a filtration $F=\{\mathcal{F}_{t}\}_{0\leq t\leq T}$ with
$\mathcal{F}_{T}=\mathcal{F}$. Let $\mathcal{P}$ denote the set of equivalent
martingale measures. we assume that $\mathcal{F}_{0}$ is trivial and
$\mathcal{P\neq\emptyset}$. For an initial investment $X_{0}\geq0$ and a
portfolio process $\pi$ such that the wealth process%
\begin{equation}
X_{t}=X_{0}+\int_{0}^{t}\pi_{s}dS_{s}\;\;\forall t\in\lbrack0,T]
\label{wealth equation}%
\end{equation}
is well defined. A strategy $(X_{0},\pi)$ is called admissible if the
corresponding wealth process $X$ is nonnegative. For a given nonnegative
contingent claim $H\in L^{\infty}(\mu)$, we define that
\[
U_{0}=\underset{P\in\mathcal{P}}{\sup}E_{P}[H].
\]
It is well known that if the investor's initial wealth $\tilde{X}_{0}<U_{0}$,
then some shortfall $(H-X_{T})^{+}$ will occur at time $T$.

In this section, we introduce a general convex expectation $\rho$ to measure
the shortfall $(H-X_{T})^{+}$.

\begin{definition}
For a given convex expectation $\rho$, the shortfall risk is defined as
\[
\rho((H-X_{T})^{+})\text{.}%
\]

\end{definition}

Consequently, the investor wants to find an admissible strategy $(X_{0},\pi)$
which minimizes the shortfall risk and control his initial investment
$X_{0}\leq\tilde{X}_{0}$. Thus, we will solve the following optimization
problem:%
\begin{equation}%
\begin{array}
[c]{c}%
\underset{(X_{0},\pi)}{\text{min}}\;\rho((H-X_{T})^{+}),\\
\text{subject to }X_{0}\leq\tilde{X}_{0},
\end{array}
\label{convex risk optimization-1}%
\end{equation}
where $\tilde{X}_{0}$ is the initial wealth of the investor.

Now we show that the optimal $X_{T}^{\ast}$ must satisfy $0\leq X_{T}^{\ast
}\leq H$. In fact, if $P(X_{T}^{\ast}>H)>0$, we can construct a feasible
terminal wealth $\tilde{X}_{T}$ such that $0\leq\tilde{X}_{T}\leq H$ and
$(H-\tilde{X}_{T})^{+}<(H-X_{T}^{\ast})^{+}$. Thus, $\rho((H-\tilde{X}%
_{T})^{+})<\rho((H-X_{T}^{\ast})^{+})$ by the monotonicity property of $\rho$.
This leads to a contradiction.

Thus, without loss of generality we assume that $0\leq X_{T}\leq H$ and
(\ref{convex risk optimization-1}) becomes%
\begin{equation}%
\begin{array}
[c]{c}%
\underset{0\leq X_{T}\leq H}{\text{min}}\rho(H-X_{T}),\\
\text{subject to }\underset{P\in\mathcal{P}}{\sup}E_{P}[X_{T}]\leq\tilde
{X}_{0}.
\end{array}
\label{convex risk optimization-2}%
\end{equation}

By Theorem \ref{main-result} and the classical Neyman-Pearson lemma
(Proposition 4.1 in \cite{fo-le-2000}),%
\[
X_{T}^{\ast}=HI_{\{zH_{Q^{\ast}}>G_{P^{\ast}}\}}+BI_{\{zH_{Q^{\ast}%
}=G_{P^{\ast}}\}},\quad\mu-a.e.,
\]
where%
\[
z=\sup\{\tilde{z}\mid\int_{\{\tilde{z}H_{Q^{\ast}}>G_{P^{\ast}}\}}HdP^{\ast
}\leq\tilde{X}_{0}\}
\]
and%
\[
B=\left\{
\begin{array}
[c]{ll}%
\frac{\tilde{X}_{0}-\int_{\{zH_{Q^{\ast}}>G_{P^{\ast}}\}}HdP^{\ast}}%
{\int_{\{zH_{Q^{\ast}}=G_{P^{\ast}}\}}HdP^{\ast}}, & \;\;\text{when }P^{\ast
}[\{H>0\}\cap\{zH_{Q^{\ast}}=G_{P^{\ast}}\}]>0;\\
0, & \;\;\text{otherwise.}%
\end{array}
\right.
\]

Then by the optional decomposition theorem (see \cite{kra} and \cite{fo-kab}),
we obtain the optimal strategy $(\tilde{X}_{0},\pi^{\ast})$ corresponding to
$X_{T}^{\ast}$.

\begin{remark}
Instead of minimizing the convex risk measure under the initial investment
constraint, we can solve the following essentially equivalent problem: fix a
bound on the convex risk measure and minimize the initial investment.
\end{remark}

\begin{remark}
We assume that the given nonnegative contingent claim $H\in L^{\infty}(\mu)$.
If $H\in L^{1}(\mu)$, then we can use Theorem \ref{auxiliary-theorem} in the appendix.
\end{remark}

\section{Appendix}

In this appendix, we first prove that when the convex expectations are
continuous from above, Assumption \ref{assumption} holds naturally. Then an
example is given to show that Assumption \ref{assumption} is only a sufficient
condition for the existence of $Q^{\ast}$. Finally, we give the Neyman-Pearson
lemma for convex expectations on $L^{1}(\mu)$.

\begin{definition}
\label{continuous from above} We call a convex expectation $\rho$ is
continuous from above iff for any sequence $\{X_{n}\}_{n\geq1}\subset
L^{\infty}(\mu)$ decreases to some $X\in L^{\infty}(\mu)$, then $\rho
(X_{n})\to\rho(X)$.
\end{definition}

\begin{proposition}
\label{a-result-from above} If $\rho_{1}$ and $\rho_{2}$ are continuous from
above, then Assumption \ref{assumption} holds.
\end{proposition}

\begin{proof}
We only show the result holds for $\rho_{1}$.

For any $u>\max\{0,M-\rho_{1}(0)+1\}$, we have $u>\max\{0,-\rho_{1}(0)\}$. By
Theorem 3.6 in \cite{r13}, $\mathcal{G}_{u}$ is uniformly integrable. For any
sequence $\{G_{P_{n}}\}_{n\geq1}\subset\mathcal{G}_{u}$ that converges to
$G_{\hat{P}}$, $\mu$-a.e., since $\{G_{P_{n}}\}_{n\geq1}$ is uniformly
integrable,%
\[
E_{\mu}[G_{\hat{P}}]=\lim\limits_{n\rightarrow\infty}E_{\mu}[G_{P_{n}}]=1,
\]
which shows $\hat{P}\in\mathcal{M}$. On the other hand, for any $u>\max
\{0,M-\rho_{1}(0)+1\}$, by Lemma \ref{Fatou-property-1}, we have
\[
\rho^{\ast}(\hat{P})\leq\liminf\limits_{n\rightarrow\infty}\rho^{\ast}%
(P_{n})\leq u.
\]
Then $G_{\hat{P}}\in\mathcal{G}_{u}$. Thus, $\mathcal{G}_{u}$ is closed under
the $\mu$-a.e. convergence.
\end{proof}

Now we show that even if Assumption (\ref{assumption}) does not hold, the
probability measure $Q^{\ast}$ may still exist.

\begin{example}
\label{example-continuous} Consider the probability space $(\Omega
,\mathcal{B},\mu)$, where $\Omega$ is the interval $[0,1]$, $\mathcal{B}$ is
the collection of all Borel sets in $[0,1]$ and $\mu$ is the Lebesgue measure
defined on $[0,1]$. Set $K_{1}=0$, $K_{2}=1$, $\alpha=\frac{3-e}{e-1}$,
$\rho_{1}(X)=E_{P}[X]$ and $\rho_{2}(X)=\ln E_{\mu}[e^{X}]$, where
\[
\frac{dP}{d\mu}=\Bigg\{%
\begin{array}
[c]{l@{}c}%
\frac{e+1}{e-1}, & \quad\omega\in\lbrack0,\frac{e-2}{e-1}],\\
\frac{3-e}{e-1}, & \quad\omega\in(\frac{e-2}{e-1},1].
\end{array}
\]
To solve problem (\ref{initial-problem}), one can check that\ Assumption
\ref{assumption} does not hold. Let
\[
X^{\ast}=I_{(\frac{e-2}{e-1},1]}\text{ \ and\ }\frac{dQ^{\ast}}{d\mu}=\Bigg\{%
\begin{array}
[c]{l@{}c}%
\frac{e}{e-1}, & \quad\omega\in\lbrack0,\frac{e-2}{e-1}],\\
\frac{1}{e-1}, & \quad\omega\in(\frac{e-2}{e-1},1].
\end{array}
\]
By the classical Neyman-Pearson lemma, through simple calculations, we can obtain that $X^{\ast}$ is
also the optimal test for discriminating between probability measures $P$ and
$Q^{\ast}$, i.e.,%
\[
E_{Q^{\ast}}[1-X^{\ast}]=\inf_{X\in\mathcal{X}_{\alpha}}E_{Q^{\ast}}[1-X].
\]
Furthermore, 
\[
\inf_{X\in\mathcal{X}_{\alpha}}E_{Q^{\ast}}[1-X]-\rho_{2}^{\ast}(Q^{\ast
})=\inf_{X\in\mathcal{X}_{\alpha}}\rho_{2}(1-X),
\]
\end{example}

If $K_{1}$ and $K_{2}$ belong to $L^{1}(\mu)$ such that $0\leq K_{1}<K_{2}$,
for two finite convex expectations $\rho_{1}$ and $\rho_{2}$, consider the
following problem:
\begin{equation}
\text{minimize}\quad\rho_{2}(K_{2}-X), \label{L-1-problem}%
\end{equation}
over the set $\mathcal{X}_{\alpha}=\{X:K_{1}\leq X\leq K_{2},\rho_{1}%
(X)\leq\alpha,X\in L^{1}(\mu)\}$. We find that similar ideas for solving
problem (\ref{Priliminary-1}) can be used to solve problem
(\ref{Priliminary-2}). So we obtain the following theorem and only give a
brief proof.

\begin{theorem}
\label{auxiliary-theorem} If $\rho_{1}$ and $\rho_{2}$ are two finite convex
expectations defined on $L^{1}(\mu)$ space, then the optimal test of
(\ref{L-1-problem}) exists and has the same form as in Theorem
\ref{main-result}.
\end{theorem}

\begin{proof}
Since $\rho_{1}$ and $\rho_{2}$ are finite, then they are Lebesgue-continuous.
Repeating the proof of Theorem {\ref{existence}}, we will get the optimal test
exists. On the other hand, $\rho_{1}$ and $\rho_{2}$ can be represented by
some probability set $\mathcal{P}$ and $\mathcal{Q}$ for which the density
sets $\{G_{P}\in L^{\infty}(\mu):P\in\mathcal{P}\}$ and $\{H_{Q}\in L^{\infty
}(\mu):Q\in\mathcal{Q}\}$ are weakly compact. The property of this
representation reduces the difficulty of the problem. Then, the form in
Theorem \ref{main-result} can also be obtained by the same method as in
section 4. The detailed proofs are omitted.
\end{proof}

\end{document}